\documentclass[12pt, reqno]{amsart}
\usepackage{amsmath}
\usepackage{amssymb}
\usepackage{amsfonts}
\usepackage{amsthm}
\usepackage{amscd}
\usepackage{relsize}
\usepackage{setspace}
\usepackage{geometry}
\usepackage{enumitem}
\usepackage{url}
\usepackage{xspace}
\usepackage{mathrsfs}
\usepackage{dsfont}
\usepackage{lmodern}
\usepackage{xcolor}
\usepackage[utf8]{inputenc}
\usepackage{mathtools}

\usepackage{hyperref}

\usepackage[textsize=footnotesize,textwidth=15ex]{todonotes}

\newtheorem{thm}{Theorem}[section]
\newtheorem{lem}[thm]{Lemma}
\newtheorem{prop}[thm]{Proposition}
\newtheorem{cor}[thm]{Corollary}

\theoremstyle{definition}
\newtheorem{remark}[thm]{Remark}

\numberwithin{equation}{section}

\newcommand{\ZZ}{\mathbf{Z}}

\DeclareMathOperator{\Aut}{Aut}

\DeclareMathOperator{\SL}{SL}

\begin{document}

\title[Ascending chains of irreducible lattices]{Ascending chains of irreducible lattices, bi-reversible automata and \\
affine arithmetic groups}

\author{Pierre-Emmanuel \textsc{Caprace}}
\author{Justin \textsc{Vast}}

\date{July 16, 2026}

\thanks{JV is a F.R.S.-FNRS Research Fellow; PEC and JV are supported in part by the FWO and the F.R.S.-FNRS under the EOS programme (project ID 40007542).}

\begin{abstract}
For each $n \geq 2$, we construct an ascending chain of irreducible lattices in the product of $n$ homogeneous trees. Moreover, for each pair of integers $m_1, m_2 \geq 1$, we define explicitly  a bi-reversible automaton $\mathcal B$ such that the group $G_{\mathcal B}$ defined by the automaton $\mathcal B$ has finiteness length $m_1$ (i.e. it is of type $\mathrm{F}_{m_1}$ but not of type $\mathrm{FP}_{m_1+1}$), and the group $G_{\mathcal B^*}$ defined by the dual automaton has finiteness length $m_2$. Both constructions rely on the consideration of $S$-arithmetic groups in the affine group of a global function field. 
\end{abstract}

\maketitle

%
%
%
%
%

\section{Introduction}

Let  $G$ be a unimodular locally compact group, with Haar measure $\mu$. We recall that, given a discrete subgroup $\Gamma \leq G$, there is a unique $G$-invariant measure $\nu$ on the coset space $G/\Gamma$ such that $\int_G f(g) d\mu(g) = \int_{G/\Gamma} \sum_{\gamma \in \Gamma} f(x\gamma) d\nu(x\Gamma)$ for all continuous compactly supported function $f$ on $G$. The measure $\nu(G/\Gamma)$ is then called the \textbf{covolume} of $\Gamma$ in $G$. If $\nu(G/\Gamma)$ is finite, the discrete subgroup $\Gamma$ is called  a \textbf{lattice}.  The set of covolumes of lattices have been studied in many classes of locally compact groups. We refer to Lubotzky's survey \cite{Lub} for references and a review of many other related results. Here, let us merely recall that, in semisimple Lie groups, covolumes of lattices are bounded away from zero by a classical result of Kazhdan--Margulis. A similar result holds for semisimple algebraic groups over local fields. By contrast, if $G$ is the automorphism group of a homogeneous locally finite tree of degree $k \geq 3$, Bass--Kulkarni \cite{BK90} have shown that $G$ contains strictly ascending chains of cocompact lattices. This implies that the set of covolumes of lattices has~$0$ as an accumulation point. This has been generalized to the isometry groups of right-angled buildings by A.~Thomas \cite{Tho06} (see also \cite[Corollary~7.3]{Hughes} for the case of the universal cover of Salvetti complexes). 

On the other hand, if $G = G_1 \times \dots \times G_n$ is a product of automorphism groups of     trees, it has been shown by various authors (see \cite{Gla03}, \cite{BM14}, \cite{CLB}) that, under suitable conditions, ascending chains of \textit{irreducible} lattices cannot exist (see also \cite[\S1.1]{CLB} for covolume bounds that  apply to an abstract framework of cocompact lattices in products of compactly generated almost simple groups). The notion of \textbf{irreducibility} means that the projection of the lattice $\Gamma \leq G$ to each factor $G_i$ is non-discrete.   However, the question whether there can exist ascending chains of irreducible lattices in products of trees had remained open. Our first main result answers this positively. We denote by $T_k$ the homogenous tree of degree~$k$. 

\begin{thm}[See Corollary~\ref{cor:asc-chains-trees:2}]\label{thm:exist-chains}
Let $n \geq 2$ be an integer and $q$ be a prime power. There is a sequence $(q_1 = q, q_2, \dots, q_n)$ of powers of $q$ such that the product $\Aut(T_{2q_1}) \times \dots \times \Aut(T_{2q_n})$ contains an infinite strictly ascending chain of cocompact vertex-transitive lattices with non-discrete projections on any sub-product. 
\end{thm}

It has been conjectured by Y.~Glasner \cite[Conj.~1.5]{Gla03} that there is a positive lower bound on the covolume of irreducible lattices  in $\mathrm{Aut}(T_1) \times \mathrm{Aut}(T_2)$ with primitive local actions. This was confirmed for vertex-transitive lattices with $2$-transitive local actions, see  \cite[Th.~E]{CLB}. Theorem~\ref{thm:exist-chains} shows that a lower bound on the covolume does not exist if one removes any hypothesis on the local actions. 

Our construction relies on considering $S$-arithmetic lattices in affine groups over a global function field $K$. By choosing a suitable set $S$ that is the disjoint union of $n$ finite sets $S_1, \dots, S_n$ of at least two places of $K$, we obtain a normal subgroup $\Lambda_1$ of the affine group $\mathcal O_S \rtimes \mathcal O_S^*$ that embeds an irreducible cocompact lattice in the product $H= \mathrm{Aut}(L_1) \times \dots \times \mathrm{Aut}(L_n)$ of automorphism groups of homogeneous graphs that are \textit{Diestel--Leader graphs} (see \cite{BNW08} or Section~\ref{sec:DL} below for the definition). Using a place that does not belong to $S$, we  build a strictly ascending chain of lattices $\Lambda_1 < \Lambda_2 < \dots$ in the same group $H$. We then replace the graph $L_i$ by its universal covering tree $T_i$. The chain $\Gamma_1 < \Gamma_2 < \dots$ is obtained by lifting $\Lambda_1 < \Lambda_2 < \dots$ up to the universal cover. 

In the case of the global field $K = F(t)$, where $F$ is a finite field, we may arrange that $\Lambda_1$ acts \textbf{regularly} (i.e. sharply transitively)  on the vertices of the product $L_1 \times \dots \times L_n$. Moreover, the subgroup fixing a vertex in all but one factors is commensurable to the affine arithmetic group $\mathcal O_{S_i} \rtimes \mathcal O_{S_i}^*$. The homological finiteness properties of those groups have been studied by K.-U. Bux in \cite{Bux, Bux04}, who showed that $\mathcal O_{S_i} \rtimes \mathcal O_{S_i}^*$ has property $\mathrm{F}_{|S_i|-1}$, but not property $\mathrm{FP}_{|S_i|}$. Putting $n=2$ and using the relation between lattices in products of two trees and bi-reversible automata established by Glasner--Mozes~\cite{GM05}, we obtain the following (see \cite{GM05, Nekra} and \S\ref{sec:automata} below for the definitions of automata and the groups associated with them). 

\begin{thm}[See Theorem~\ref{thm:bi-rev-autom:full}]\label{thm:bi-rev-autom}
Let $m_1, m_2 \geq 1$ be integers. There exists a bi-reversible automaton   $\mathcal B$ such that the automaton group $G_{\mathcal B}$   is of type $\mathrm{F}_{m_1}$ but not of type $\mathrm{FP}_{m_1+1}$, and the group $G_{\mathcal B^*}$ defined by the dual automaton is of type $\mathrm{F}_{m_2}$ but not $\mathrm{FP}_{m_2+1}$. 
\end{thm}

For each prime power $q \geq m_1+m_2+2$, we   describe explicitly a bi-reversible automaton  $\mathcal B= \mathcal B_{m_1, m_2; q}$ satisfying the conclusions of Theorem~\ref{thm:bi-rev-autom}, see  \S\ref{sec:automata}. That automaton $\mathcal B_{m_1, m_2; q}$ has a set of states of size $\frac{m_1(m_1+1)} 2 q$, and an alphabet of size $\frac{m_2(m_2+1)} 2 q$. 
	
Little is known about the class of groups that can be defined by a bi-reversible automaton. The first example of a bi-reversible automaton $\mathcal A$ such that $G_{\mathcal A}$ is finitely generated but not finitely presented was obtained by Bondarenko--D'Angeli--Rodaro \cite{BDR}, who constructed explicitly a bi-reversible automaton $\mathcal A$ such that $G_{\mathcal A}$ is the lamplighter group $C_3 \wr \ZZ$.  It was shown that Skipper--Steinberg \cite{SkSt20} and Francoeur~\cite{Fra23} that for each finite abelian group $A$, the wreath product $A \wr \ZZ$ can be defined by a bi-reversible automaton. Our interest in those bi-reversible automata and the associated lattices grew out of the discovery of their fractal properties, described and studied in \cite{CJ}. 

Let us also mention   that free groups can be defined by bi-reversible automata; it is also the case of some groups with Kazhdan's property (T), see \cite{GM05}. Theorem~\ref{thm:bi-rev-autom} adds new specimens to that class. 

\subsection*{Acknowledgements}

We thank Stefan Witzel for drawing our attention to the work of K.-U. Bux~\cite{Bux04} on the finiteness properties of solvable $S$-arithmetic groups.  We also thank Adrien Le Boudec for his comments on an earlier version of this paper.
%
%

\section{Additive and multiplicative $S$-arithmetic groups}

We start by reviewing a few facts concerning the algebraic number theory of global function fields. We refer to the books \cite{Rosen, Weil} for more details. 

Let $K$ be a global function field of characteristic $p$, and $F$ be its field of constants. We denote by $\mathscr P$ the set of all places of $K$, that we view as normalized discrete valuations $v \colon K \to \ZZ \cup\{\infty\}$.

Given a valuation $v \in \mathscr P$, we denote by $K_v$ the corresponding local field,  by $\mathcal R_v  = \{x \in K_v \mid v(x) \geq 0\}$ its valuation ring and by $\bar K_v$ its residue field. 

We define the \textbf{degree} $ \deg(v)$ of the valuation $v$ of $K$ by setting  $\deg(v) = [\bar K_v : F]$. 

Let  $S \subset \mathscr P$ be a non-empty finite set of places of $K$. We set 
$$\mathcal O_S = \{x  \in K \mid v(x) \geq 0 \text{ for all valuations } v \not \in S\}.$$
 
\begin{prop}\label{prop:additive-lattice}
For any (possibly empty) set of valuations $S'$ disjoint from $S$, the diagonal map
$$\mathcal O_S \to A = \prod_{v \in S} K_v \times \prod_{v' \in S'} \mathcal R_{v'}$$
embeds $\mathcal O_S$ as a cocompact lattice in the additive group $A$. 

Moreover, for any proper subset $S_0 \subset S$, the image of  $\mathcal O_S$ in the subproduct $\prod_{v \in S_0} K_v \times \prod_{v' \in S'} \mathcal R_{v'}$ is dense.

If in addition $K = F(t)$, then we have 
$$\mathcal O_S \cap R = F \qquad \text{and} \qquad  \mathcal O_S + R = A,$$ 
where $R = \prod_{v \in S \cup S'} \mathcal R_v$. 
\end{prop}
\begin{proof}
For the first assertion, see Theorem 2 in \cite[Chapter~IV, \S 2]{Weil}.  The  assertion on the density of projections follows from Corollary~2 in \cite[Chapter~IV, \S 2]{Weil}.  

In the special case where $K = F(t)$, the fact that $\mathcal O_S \cap R = F$ follows readily from the definitions, and the equality $\mathcal O_S + R = A$ follows from Lemma~4 in \cite[Chapter~IV, \S 2]{Weil}.
\end{proof}

Given a finite set $S \subset \mathscr P$, we also denote by  $\mathcal O_S^*$ the group of units of the ring $\mathcal O_S$, that coincides with
 $$\mathcal O_S^* = \{x  \in \mathcal O_S \mid \ v(x) = 0 \text{ for all valuations } v \not \in S\}.$$
We recall from  Theorem  9 in \cite[Chapter~IV, \S 4]{Weil} that  $\mathcal O_S^*$ splits as the direct product 
\begin{align}\label{eq:T_S}
\mathcal O^*_S &\cong F^* \times T_S,
\end{align}
where $T_S \leq \mathcal O^*_S$ is a free abelian group of rank~$|S|-1$.

We also set 
$$K_S^1 = \big\{(g_v)  \in \prod_{v \in S} K_v^* \mid \sum_{v \in S} \deg(v) v(g_v) = 0\big\}.$$
By the \textbf{Artin product formula} (see Theorem 5 in \cite[Chapter~IV, \S 4]{Weil}), the diagonal map $\mathcal O_S^* \to  \prod_{v \in S} K_v^*$ takes its values in $K_S^1$. 

\begin{prop}\label{prop:mult-lattice}
The diagonal map
$$\mathcal O_S^* \to  \prod_{v \in S} K_v^*$$
embeds 
 $\mathcal O_S^*$ as a cocompact lattice in the subgroup  $K_S^1$ of the multiplicative group $\prod_{v \in S} K_v^*$. Moreover, setting $R_S = \prod_{v \in S} \mathcal R_v^* \leq K_S^1$, we have 
$\mathcal O_S^* \cap R_S = F^*$.

If in addition $K = F(t)$, then we have  $\mathcal O_S^* R_S = K_S^1$. 
\end{prop}
\begin{proof}
The first assertion follows from Theorem  6 in \cite[Chapter~IV, \S 4]{Weil}. The fact that $\mathcal O_S^* \cap R_S = F^*$ is a  consequence of the definitions. 

The equality $\mathcal O_S^* R_S = K_S^1$ means that the group $\mathcal O_S^*$  maps surjectively on the quotient group $K_S^1/R_S$. The fact that this is indeed the case 
in the special case where $K = F(t)$ can be established by a direct computation, see Theorem~9 and its corollary in \cite[Chapter~IV, \S 4]{Weil}. 

More conceptually, the quotient   $K_S^1/\mathcal O_S^*R_S$ can naturally be identified with the group $\mathcal D(S)^0/\mathcal P(S)$ appearing in \cite[Prop.~14.1(b)]{Rosen}. The group $\mathcal P(S)$ is the subgroup of the group $\mathrm{Div}(K)$ of divisors of $K$ generated by the principal divisors in $\mathcal O_S^*$, while $\mathcal D(S)^0$ consists of all the degree zero divisors contained in the group of divisors generated by the places in $S$. The short exact sequence afforded by the latter result implies that $\mathcal D(S)^0/\mathcal P(S)$ is trivial as soon as the function field $K$ has class number one (i.e. its group of divisor classes of degree zero is trivial). This is indeed the case for the global field $K = F(t)$. (Note that $\mathcal D(S)^0/\mathcal P(S)$ need  not be   trivial in general.)  
\end{proof}

\section{Affine $S$-arithmetic groups}

Given a field $k$, we define the \textbf{affine group} of $k$ as the semi-direct product $k \rtimes k^*$, where the action of $k^*$ on $k$ is by multiplication. We denote it by $\mathrm{AGL}_1(k)$. 

We now continue with the notation of the previous section. 
We set
 $$\Gamma_S = \mathcal O_S \rtimes \mathcal O_S^*.$$
This is a subgroup of the affine group $\mathrm{AGL}_1(K)$ of the global field $K$. 
By (\ref{eq:T_S}), the group of units $\mathcal O_S^*$ splits as the direct product of the multiplicative group $F^*$ and a free abelian group $T_S$ of rank~$|S|-1$.
%
%
%

We also set $A_S = \prod_{v \in S} K_v$, and 
$$G_S = A_S \rtimes K_S^1 \leq \prod_{v \in S} \mathrm{AGL}_1(K_{v}).$$
Combining Propositions~\ref{prop:additive-lattice} and~\ref{prop:mult-lattice}, we obtain the following direct consequence.  

\begin{prop}\label{prop:affine-arith}
The diagonal map $\Gamma_S \to \prod_{v \in S} \mathrm{AGL}_1(K_{v})$ embeds $\Gamma_S$ as a cocompact lattice in $G_S$. Moreover, if $|S| \geq 2$, the projection of $\Gamma_S$ to each factor $ \mathrm{AGL}_1(K_{v})$ is non-discrete. 
\end{prop}

The finiteness properties of the affine arithmetic group $\Gamma_S =  \mathcal O_S \rtimes \mathcal O_S^*$   have been analyzed in detail by K.-U. Bux \cite{Bux, Bux04}. The following result follows from his work.

\begin{thm}\label{thm:Bux}
The affine $S$-arithmetic group  $\Gamma_S =  \mathcal O_S \rtimes \mathcal O_S^*$ is of type $\mathrm{F}_{|S|-1}$, but not of type $\mathrm{FP}_{|S|}$. 
\end{thm}
\begin{proof}
The second theorem in the introduction of \cite{Bux} ensures that $\Gamma_S$ is of type $\mathrm{FP}_{|S|-1}$, but not of type $\mathrm{FP}_{|S|}$. In \cite[Cor.~3.5]{Bux04}, it is proved that the Borel subgroup of $\SL_2$ over $\mathcal O_S$ is  actually of type $\mathrm{F}_{|S|-1}$. That Borel subgroup is commensurable with $\Gamma_S$ (see the exact sequences in \cite[Remark~3.6]{Bux04}), hence they have the same finiteness properties, see \cite{Alonso}. 
\end{proof}

In particular $\Gamma_S$ is finitely generated if and only if $|S| \geq 2$.  For the sake of completeness, we include a direct proof of this fact. 

\begin{lem}\label{lem:compact-gen}
The following assertions are equivalent. 
\begin{enumerate}[label=(\roman*)]
\item  $\Gamma_S$ is finitely generated.
\item $G_S$ is compactly generated. 
\item  $|S|\geq 2$. 
\end{enumerate}
\end{lem}
\begin{proof}
The equivalence between (i) and (ii) is a consequence of Proposition~\ref{prop:affine-arith}, see \cite[Proposition~4.C.11(2)]{CoHa}. 

If $S =\{v\}$ then $G_S = K_v \rtimes K^1_v$, where $K^1_v = \{x \in K_v^* \mid v(x) =0\}$. Observing that $ K^1_v$ is compact, it follows that $G_S$ contains the additive group $K_v$ as a cocompact subgroup. Since the latter is the union of an infinite properly ascending chain of compact open subgroups, it is not compactly generated. Hence $G_S$ is not compactly generated either (see \cite[Proposition~4.C.11(2)]{CoHa}). 

Suppose now that $|S|\geq 2$. The group $K^1_S$ contains the free abelian group $T_S$ as a cocompact lattice by Proposition~\ref{prop:mult-lattice}, hence it is compactly generated. Hence the subgroup $H$ of $G_S$ generated by $K_S^1$ and the compact subgroup $\prod_{v \in S} \mathcal R_v$ is itself compactly generated. For each $v$, the projection of $H$ to $\mathrm{AGL}_1(K_{v})$ contains a cocompact subgroup of $K_v^*$, hence some positive power of a uniformizer, say $\pi_v \in K_v^*$. It follows that $H$ contains the additive group $K_v = \bigcup_{n >0} \pi_v^{-n} \mathcal R_v$. Hence $H$ contains $\prod_{v \in S} K_v$, so that $H = G_S$. 
\end{proof}

\section{Ascending chains}\label{sec:towers:irred}

We retain the notation of the previous sections	and we now choose  a discrete valuation $v_0$ of $K$ not contained in $S$. For each integer $k  \in \mathbf Z$, set 
$$\mathcal O_{S, k} = \{x  \in \mathcal O_{S \cup \{v_0\} }\mid v_0(x) \geq -k\}.$$
Thus we have an ascending  chain 
$$\mathcal O_S = \mathcal O_{S, 0} < \mathcal O_{S, 1} < \dots  $$ 
where each quotient group $\mathcal O_{S, k+1}/\mathcal O_{S, k}$ is finite, and isomorphic to the additive group of $\bar K_{v_0}$. For $x \in \mathcal O_S^*$, we have $v_0(x) = 0$ so that the multiplication by $x$ preserves $\mathcal O_{S, k}$ for all $k$. Hence, we may set 
$$\Gamma_k = \mathcal O_{S, k} \rtimes \mathcal O_S^*.$$

\begin{cor}\label{cor:affine:chain}
The groups $\Gamma_S = \Gamma_0  < \Gamma_1 < \dots$ form an ascending chain of cocompact lattices in $G_S$.
\end{cor}
\begin{proof}
The discussion above implies that   $[\Gamma_{k+1} : \Gamma_k ] = |\bar K_{v_0}|$ for all $k$. Thus the assertion is a direct consequence of Proposition~\ref{prop:affine-arith}. 
\end{proof}

Our next goal is to extend that construction   in order to obtain ascending chains of irreducible lattices in products. For that purpose, we
 fix an integer $n \geq 2$ and a collection $S_1, \dots, S_n$ of disjoint finite sets of valuations of $K$.
We also set 
$S = \bigcup_{j=1}^n S_j$, 
and
$$\Gamma_{S_1, \dots, S_n} = \mathcal O_S \rtimes \big(\mathcal O_{S_1}^*  \dots \mathcal O_{S_n}^*\big).$$
Observe that the product $\mathcal O_{S_1}^*  \dots \mathcal O_{S_n}^*$ is a subgroup of $\mathcal O_S^*$, so that $\Gamma_{S_1, \dots, S_n}$ is a normal subgroup of $\Gamma_S$, affording the quotient $\Gamma_S/\Gamma_{S_1, \dots, S_n} \cong \mathbf Z^{n-1}$. More importantly for us, the group $\Gamma_{S_1, \dots, S_n}$ naturally embeds in $G_{S_i}$ for all $i$. The diagonal map thus yields an embedding 
$$\Gamma_{S_1, \dots, S_n} \to G_{S_1} \times \dots \times G_{S_n}.$$
We shall identify $\Gamma_{S_1, \dots, S_n} $ with its image. 

\begin{thm}\label{thm:affine-arith:prod}
The following assertions hold. 
\begin{enumerate}[label=(\roman*)]
\item The group $\Gamma_{S_1, \dots, S_n}$ is a cocompact lattice in $G = G_{S_1} \times \dots \times G_{S_n}$. 

\item The projection of $\Gamma_{S_1, \dots, S_n}$ to every sub-product of $G_{S_1} \times \dots \times G_{S_n}$ is injective, and has a non-discrete image. 
%
%
%
%
\item There is an infinite strictly ascending chain of lattices 
$$\Gamma_{S_1, \dots, S_n} = \Gamma_0 < \Gamma_1 < \dots < G.$$
\end{enumerate}
\end{thm}
	
\begin{proof}
(i) In view of Proposition~\ref{prop:additive-lattice}, it suffices to prove that the multiplicative part $\mathcal O_{S_1}^*  \dots \mathcal O_{S_n}^*$ of $\Gamma_{S_1, \dots, S_n} $ is a cocompact lattice in the multiplicative part of $\prod_i G_{S_i}$, which is the group $\prod_{i=1}^n K^1_{S_i}$. 

Since $\mathcal O_{S_1}^*  \dots \mathcal O_{S_n}^* \leq  \mathcal O^*_S \cap \big( \prod_{i=1}^n K^1_{S_i}\big)$ by definition, we deduce from Proposition~\ref{prop:mult-lattice} that $\mathcal O_{S_1}^*  \dots \mathcal O_{S_n}^*$ is a discrete subgroup of $\prod_{i=1}^n K^1_{S_i}$. Now, the compact group $R_S^* = \prod_{v \in S} \mathcal R^*_v$ is open and normal in the abelian group $\prod_{i=1}^n K^1_{S_i}$, and the quotient is free abelian of rank $\sum_{i=1} (|S_i|-1) = |S|-n$. On the other hand the product $\mathcal O_{S_1}^*  \dots \mathcal O_{S_n}^*$ is virtually free abelian of rank~$|S|-n$ as well, in view of \ref{eq:T_S}. This implies that the inclusion of $\mathcal O_{S_1}^*  \dots \mathcal O_{S_n}^*$ in $\prod_{i=1}^n K^1_{S_i}$ is cocompact.
%
%

(ii) The injectivity of the projection comes from the fact that the affine group $K \rtimes K^*$ maps injectively in the affine group $K_v \rtimes K_v^*$ over the completed field with respect to every place $v \in \mathscr P$. The non-discreteness is a direct consequence of the last assertion in Proposition~\ref{prop:additive-lattice}. 

%
%

(iii) This follows directly from the same construction as in Corollary~\ref{cor:affine:chain}, by setting $\Gamma_k  = \mathcal O_{S, k} \rtimes \big(\mathcal O_{S_1}^*  \dots \mathcal O_{S_n}^*\big)$. 
\end{proof}

Theorem~\ref{thm:affine-arith:prod} is already sufficient to obtain ascending chains of irreducible lattices in products of trees. In order to establish this, we first recall the construction of \textbf{Cayley--Abels graphs} associated with compactly generated totally disconnected locally compact groups. 

\begin{lem}\label{lem:CayAb}
Let $G$ be a locally compact group generated by a compact symmetric subset $\Sigma$,  and $U \leq G$ be a compact open subgroup. Then the  graph $\mathcal G$ with vertex set $G/U$ and edge set $\{(gU, hU) \mid h^{-1}g \in U\Sigma U\}$ is connected and of finite valency. Moreover the $G$-action on $G/U$ by left multiplication defines a continuous action of $G$ by automorphisms on $\mathcal G$, which is proper and vertex-transitive. 
\end{lem}
\begin{proof}
See \cite[Prop.~2.E.9]{CoHa}. 
\end{proof}

The following result is a first approximation of Theorem~\ref{thm:exist-chains}. 

\begin{cor}\label{cor:asc-chains-trees}
For each $n$, there is an $n$-tuple of homogeneous locally finite trees $\mathcal T_1, \dots, \mathcal T_n$ and an ascending chain of cocompact lattices in $\mathrm{Aut}(\mathcal T_1) \times \dots \times \mathrm{Aut}(\mathcal T_n)$ with a non-discrete projection on each factor. 
\end{cor}
\begin{proof}
For each $i=1, \dots, n$, let $S_i$ be a set of at least two valuations of $K$.  By van Dantzig's theorem, we may choose a compact open subgroup $U_i \leq G_{S_i}$ for each $i$, in such a way that the product $\prod_i U_i$ intersects trivially the cocompact lattice $\Gamma_{S_1, \dots, S_n}$ afforded by Theorem~\ref{thm:affine-arith:prod}.  By Lemma~\ref{lem:compact-gen}, the group $G_{S_i}$ is compactly generated. Let $\mathcal G_i$ be a  Cayley--Abels graph for $G_{S_i}$   with vertex set $G_{S_i}/U_i$, see Lemma~\ref{lem:CayAb}. The group $\Gamma_{S_1, \dots, S_n}$  acts properly and cocompactly on the cartesian product $\prod_i \mathcal G_i$ via its lattice embedding in $\prod_i G_{S_i}$. By construction that action is faithful, since $\Gamma_{S_1, \dots, S_n} \cap \prod_i U_i =\{e\}$. 

Let $\mathcal T_i$ be the universal covering tree of $\mathcal G_i$. Let also $\Lambda_k \leq \mathrm{Aut}(\mathcal T_1) \times \dots \times \mathrm{Aut}(\mathcal T_n)$ be the lift of the group $\Gamma_{S_1, \dots, S_n; k}=   \mathcal O_{S, k} \rtimes \big(\mathcal O_{S_1}^*  \dots \mathcal O_{S_n}^*\big)$. Hence the group $\Lambda_k$ fits in a short exact sequence 
$$1 \to \pi_1(\mathcal G_1) \times \dots \times \pi_1(\mathcal G_n) \to \Lambda_k \to \Gamma_{S_1, \dots, S_n; k} \to 1,$$ 
see \cite[Theorem~8.A.20]{CoHa}.
By construction the groups $\Lambda_k$ act geometrically on $\mathcal T_1 \times \dots  \times \mathcal T_n$, hence they form an ascending chain of cocompact lattices in $\mathrm{Aut}(\mathcal T_1) \times \dots \times \mathrm{Aut}(\mathcal T_n)$.  
The projection of $\Lambda_k$ to the factor  $ \mathrm{Aut}(\mathcal T_i)$ normalizes the fundamental group $\pi_1(\mathcal G_i)$, which is a discrete subgroup since its action on $\mathcal T_i$ by deck transformations is free and properly discontinuous. It follows that the projection of $\Lambda_k$ to  $ \mathrm{Aut}(\mathcal T_i)$ is discrete if and only if  the projection of the quotient $\Lambda_k/\pi_1(\mathcal G_i)$ to $G_{S_i}$ is discrete. This is not the case for $k=0$ by Theorem~\ref{thm:affine-arith:prod}, hence for any $k \geq 0$ since $\Lambda_0 \leq \Lambda_k$. 
\end{proof}

\section{Diestel--Leader graphs}\label{sec:DL}

At this point, it is not clear a priori what the degrees of the trees $\mathcal T_i$ in Corollary~\ref{cor:asc-chains-trees} can be, and whether the lattices in the chain act vertex-transitively on the product. We shall now address those questions, in order to establish a more precise version of Corollary~\ref{cor:asc-chains-trees} that will appear as Corollary~\ref{cor:asc-chains-trees:2} below. The first requires to construct a suitable geometric realization of the group $G_S$.

Let $d \geq 2$ be an integer. Following \cite{BNW08}, given pairs $(\mathcal T_1, \mathfrak h_1), \dots, (\mathcal T_d, \mathfrak h_d)$ consisting of a tree $\mathcal T_i$ and a Busemann function $\mathfrak h_i \colon V\mathcal T_i \to \ZZ$, one defines the \textbf{horocyclic product}  as the graph with vertex set 
$$\{(v_1, \dots, v_d) \in \prod_{i=1}^d V\mathcal T_i \mid \sum_{i=1}^d \mathfrak h_i(v_i) = 0\},$$ 
where $(v_1, \dots, v_d)$ is adjacent to $(w_1, \dots, w_d)$ if and only if there exist two distinct indices $i, j$ such that $v_i, w_i$ and $v_j, w_j$ are  adjacent pairs in $\mathcal T_i$ and $\mathcal T_j$ respectively, and $v_k = w_k$ for all $k \neq i, j$. 
For a $d$-tuple of integers $(q_1, \dots, q_d)$, we define the \textbf{Diestel--Leader graph}  $\mathsf{DL}(q_1, \dots, q_d)$ as the horocyclic product  of $d$ homogeneous trees of degree $q_1+1, \dots, q_d+1$ respectively. We also set  $\mathsf{DL}_d(q) = \mathsf{DL}(q, \dots, q) $ in case $q_1 = \dots = q_d$.

For each valuation $v$ of $K$, we may view the affine group $\mathrm{AGL}_1(K_{v})$  as a closed subgroup of $\mathrm{GL}_2(K_{v})$ by sending the pair $(b, a)$ to the matrix $\left(\begin{array}{cc} a & b \\ 0 & 1 \end{array} \right)$. The image intersects trivially the center of $\mathrm{GL}_2(K_{v})$, and thus defines a  faithful, continuous and proper action on the Bruhat--Tits tree $\mathcal T_{K_v}$ of $\mathrm{PGL}_2(K_{v})$, which is the regular tree of degree $q+1$, where $q = |\bar K_v|$. The action of $\mathrm{AGL}_1(K_{v})$ is vertex-transitive and fixes an end. The subgroup $\mathcal R_v \rtimes \mathcal R_v^*$ is the stabilizer of a vertex. 

Let now $S = \{v_1, \dots, v_d\}$ be  a set of $d$ valuations of $K$. Assume that $\deg(v_i) = e$ for all $i$, so that the residue field $|\bar K_v|$ has the same order $q= |F|^e$ for all $v \in S$.  Let $U_S \leq G_S$ be the compact open subgroup defined by 
$$U_S =\prod_{v \in S} (\mathcal R_v \rtimes \mathcal R_v^*).$$
The group $G_S \leq \prod_{v \in S} \mathrm{AGL}_1(K_{v})$ acts on the product $\prod_{v \in S} \mathcal T_{K_v}$ of the Bruhat--Tits trees, and   $U_S$ is the stabilizer of a vertex in that product. Arguing as in \cite[Proposition~2.5]{BNW08}, one sees that $G_S$ preserves the corresponding horocyclic product of those trees, and acts vertex-transitively (the Busemann function $\mathfrak h_v$ on the tree $\mathcal T_{K_v}$ is associated with the end fixed by the upper triangular subgroup of $\mathrm{GL}_2(K_{v})$). Thus we obtain the following result.

\begin{lem}\label{lem:DL}
Assume that $S = \{v_1, \dots, v_d\}$ consists of $d$ valuations of the same degree, and let $q = |\bar K_{v_i}|$ be the common order to corresponding residue fields. 
The group $G_S$ has a continuous, faithful, vertex-transitive action by automorphisms on the Diestel--Leader graph $\mathsf{DL}_d(q)$. The compact open subgroup $U_S \leq G_S$ is the stabilizer of a vertex.  
\end{lem}

In view of this geometric realization, Corollary~\ref{cor:affine:chain} provides an ascending chain of vertex-transitive lattices in $\Aut(\mathsf{DL}_d(q))$. Next we construct an  affine arithmetic group $\Lambda_S$, with the aim to show that it admits $\mathsf{DL}_d(q)$ as  a Cayley graph (see Proposition~\ref{prop:affine:Cayley} below). This sheds some arithmetic light on  the discussion of \cite[\S 3]{BNW08}.  

We will now focus on the   rational function field   $K = F(t)$   in the indeterminate $t$. We recall from \cite[Chapter~III, Theorem~2]{Weil} that every discrete valuation  $v \colon K \to \mathbf Z \cup \{\infty\}$ is equivalent to one of the following:
\begin{itemize}
\item The $f$-adic valuation $v_f$ associated to a monic irreducible polynomial $f \in F[t]$ of degree~$\geq 1$, determined by the condition that $v_f(f)= 1$.
\item The valuation $v_\infty$,  determined by the condition that $v_\infty(t)= -1$. 
\end{itemize}

The valuation $v_\infty$ is the unique valuation, up to equivalence, such that $v_\infty(t)< 0$. We have $v_\infty(f) =  -  \deg(f)$ for all non-zero $f \in F[t]$. 

%

\begin{prop}\label{prop:affine:Cayley}
Let $K = F(t)$ and $S \subset \mathscr P$ be a non-empty finite set of places. 
Then $G_S=  \Gamma_S U_S$ and $\Gamma_S \cap U_S \cong F\rtimes F^*$.  
In particular, if   $S$ consists of $d \geq 2$ valuations of the same degree $e$,  then $\Gamma_S $ acts vertex-transitively on $\mathsf{DL}_d(q)$, where $q = |F|^e$. 

Moreover,  if $v_0 \in \mathscr P \setminus  S$ has degree one, the group $\Lambda_S = \mathcal O_{S, -1} \rtimes T_S$ (where $ \mathcal O_{S, -1}$ is defined as in \S\ref{sec:towers:irred} and $T_S$ is as in (\ref{eq:T_S})), acts regularly on the vertices of $\mathsf{DL}_d(q)$. In particular $\Lambda_S $ has $\mathsf{DL}_d(q)$ as a Cayley graph. 
\end{prop}
\begin{proof}
The first  assertions follow from Propositions~\ref{prop:additive-lattice} and~\ref{prop:mult-lattice}. The second assertion is then a consequence of Lemma~\ref{lem:DL}. 

Notice that $F$ is a subgroup of $\mathcal O_{S}$ that has a trivial intersection with $\mathcal O_{S, -1}$ since $v_0(x) = 0$ for all $x \in F$. Thus $\mathcal O_{S, -1} + F \cong \mathcal O_{S, -1} \oplus F$. Since $\deg(v_0) = 1$, we have $[\mathcal O_{S}: \mathcal O_{S, -1}]= |F|$.  Therefore $\mathcal O_{S} = \mathcal O_{S, -1} + F$. Set $R = \prod_{v \in S} \mathcal R_v$. 

By Proposition~\ref{prop:additive-lattice} we have $A_S  = \mathcal O_S + R = \mathcal O_{S, -1} + F + R = \mathcal O_{S, -1} + R$ since $F $ is contained in $R$. Moreover $O_{S} \cap R = F$, so that $O_{S, -1} \cap R = \{0\}$. 

Set $R_S^* = \prod_{v \in S} \mathcal R_v^* \leq K_S^1$. 
Since $F^* \leq R_S$, Proposition~\ref{prop:mult-lattice} yields $K_S^1 = \mathcal O_S^* R_S = T_SF^* R_S = T_S R_S$ and $T_S \cap R_S^*  = \{1\}$. 

Combining those observations, we infer that $\Lambda_S \cap U_S =\{e\}$ and $\Lambda_S U_S = G_S$. The conclusion follows.
\end{proof}

\begin{remark}\label{rem:explicit-gen-set}
In the special case where   $S = \{v_{t+a} \mid a \in F_0\}$ consists of   $f$-adic valuations defined by degree one polynomials of the form $f= t+a$, where $a$ belongs to some fixed subset $F_0 \subseteq F$, and where the degree one valuation $v_0 \not \in S$ required by Proposition~\ref{prop:affine:Cayley} is defined to be $v_\infty$, it is not difficult to find explicitly the set of those elements of  $\Lambda_S$ that send the base vertex of  $\mathsf{DL}_d(q)$ to each of its neighbors: it is given by
$$\Sigma_S =  \big\{(\frac x {t+b}, \frac{t+a}{t+b}) \mid x \in F, \  a \neq b \in F_0\big\},$$
where the elements of $\Lambda_S \leq K \rtimes K^*$ are represented by ordered pairs $(y, z)$. Since 
$$(\frac x {t+b}, \frac{t+a}{t+b})^{-1} =  (\frac {-x} {t+a}, \frac{t+b}{t+a}),$$
the generating set $\Sigma_S$ is symmetric. Thus   Proposition~\ref{prop:affine:Cayley} implies that the Cayley graph of $\Lambda_S$ with respect to the generating set $\Sigma_S$ is isomorphic to $\mathsf{DL}_d(q)$ . This can also be derived from \cite[Theorem~3.6]{BNW08} (using the fact that $\Aut(K)$ acts $2$-transitively on the degree one valuations of $K$). 
\end{remark}

\begin{remark}\label{rem:lamplighter}
Retain the notation of Remark~\ref{rem:explicit-gen-set}, and assume in addition that $|F_0| = 2$. Then   the group $\Lambda_S$ is isomorphic to the lamplighter group $F \wr \ZZ$, see \cite[Corollary~(3.14)]{BNW08}. 
\end{remark}

We may now extend the construction in Proposition~\ref{prop:affine:Cayley} to obtain lattices acting regularly on products of Diestel--Leader graphs. 
	
\begin{prop}\label{prop:regular-action-prod}
Let $K = F(t)$ and $S_1, \dots, S_n \subset \mathscr P$ be disjoint finite sets each consisting of at least two valuations that are all of the same degree $e$. Set $q = |F|^e$ and $d_i = |S_i|$ for all $i$. Let $S = \bigsqcup_{i=1}^n S_i$ and let also $v_0  \in \mathscr P \setminus S$ be a degree one valuation not contained in $S$. Then the following assertions hold. 
\begin{enumerate}[label=(\roman*)]
\item The group $\Lambda_{S_1, \dots, S_n} = \mathcal O_{S, -1} \rtimes \big( T_{S_1} \dots T_{S_n} \big) \leq \prod_{i=1}^n \Aut(\mathsf{DL}_{d_i}(q))$ acts regularly on the vertices of $\prod_{i=1}^n \mathsf{DL}_{d_i}(q)$. 
\item The projection of $\Lambda_{S_1, \dots, S_n}$ to every subproduct $\prod_{i \neq k} \mathrm{Aut}(\mathsf{DL}_{d_i}(q))$ is non-discrete for all $k$. 
\item For each $k \in \{1, \dots, n\}$, the stabilizer of  the $(n-1)$-tuple  $(v_1, \dots , \hat{v_k}, \dots, v_n)$ is isomorphic to $\Lambda_{S_k}$.  

\end{enumerate}
	
\end{prop}
\begin{proof}
By construction, the group $\Lambda_{S_1, \dots, S_n}$ is contained as a finite index subgroup in the group 
$\Gamma_{S_1, \dots, S_n} = \mathcal O_S \rtimes \big(\mathcal O_{S_1}^*  \dots \mathcal O_{S_n}^*\big)$ considered in Section~\ref{sec:towers:irred}. In particular $\Lambda_{S_1, \dots, S_n}$ is a cocompact lattice in $\prod_{i=1}^n G_{S_i}$, see Theorem~\ref{thm:affine-arith:prod}.    The group $G_{S_i}$ acts vertex-transitively on $\mathsf{DL}_{d_i}(q)$, and the vertex stabilizers are conjugate to $U_{S_i}$. Hence $\prod_{i=1}^n G_{S_i}$ acts vertex-transitively on $\prod_{i=1}^n \mathsf{DL}_{d_i}(q)$, and vertex stabilizers are conjugate to $\prod_{i=1}^n U_{S_i} = U_S$. Now the assertion~(i) follows as   in the proof of Proposition~\ref{prop:affine:Cayley}, and the assertion~(ii) is immediate from Theorem~\ref{thm:affine-arith:prod}(ii). 

For (iii), we need to show that the intersection $\Lambda_{S_1, \dots, S_n} \cap \big(\prod_{i\neq k} U_{S_i} \big)$ coincides with $\Lambda_{S_i}$. One sees that this is indeed the case by considering separately the additive and the multiplicative parts. 
\end{proof}



We let $T_k$ denote the $k$-regular tree. The following result is a version of Corollary~\ref{cor:asc-chains-trees} that brings more precise information on the degrees of the trees and ensures vertex-transitivity. 

\begin{cor}\label{cor:asc-chains-trees:2}
Let $q$ be a prime power. 
\begin{enumerate}[label=(\roman*)]
\item For each $n \geq 1$, there is a sequence $(q_1 = q, q_2, \dots, q_n)$ of powers of $q$ such that the product $\Aut(T_{2q_1}) \times \dots \times \Aut(T_{2q_n})$ contains an infinite strictly ascending chain of cocompact vertex-transitive lattices with non-discrete projections on any sub-product.

\item If $q \geq 3$, then the product $\Aut(T_{2q}) \times \Aut(T_{2q})$ contains an infinite strictly ascending chain of cocompact vertex-transitive lattices with non-discrete projections on both factors. 

\end{enumerate}
\end{cor}
\begin{proof}
Let $F$ be the field of order $q$ and $K = F(t)$. 

(i) We choose a pair $S_1$ of distinct valuations of degree $1$ and, for each $i=2, \dots, n$, a pair $S_i$ of distinct valuations of the same degree $e_i$, in such a way that $S_1, \dots, S_n$ are  pairwise disjoint.

By Theorem~\ref{thm:affine-arith:prod}, the product $G_{S_1} \times \dots \times G_{S_n}$ has an ascending chain of lattices $\Gamma_0 < \Gamma_1 < \dots$ with non-discrete projections.  By Lemma~\ref{lem:DL}, the group $G_{S_i}$ has a continuous, faithful, proper and vertex-transitive action on the Diestel--Leader graph $\mathsf{DL}_2(q_i)$, where $q_i = q^{e_i}$.  Since $\mathsf{DL}_2(q_i)$  is regular of degree $2q_i$, its universal cover is the regular tree  $T_{2q_i}$. Moreover $\Gamma_0 = \Gamma_{S_1, \dots, S_n}$ contains the lattice $\Lambda_{S_1, \dots, S_n}$ from Proposition~\ref{prop:regular-action-prod}, and is thus vertex-transitive on the product $\prod_{i=1}^n \mathsf{DL}_2(q_i)$. The conclusion follows as in Corollary~\ref{cor:asc-chains-trees} by lifting each lattice $\Gamma_k \leq G_{S_1} \times \dots \times G_{S_n} \leq \Aut(\mathsf{DL}_2(q_1)) \times \dots \times \Aut(\mathsf{DL}_2(q_n))$ to the universal covering trees.

(ii)
The condition that $q \geq 3$ ensures that $K$ possesses at least $4$ discrete valuations of degree~$1$. Thus it is possible to take $n=2$ and $e_2 = 1$ in the construction from (i). Assertion (ii) follows.
\end{proof}

\begin{remark}
The proof of Corollary~\ref{cor:asc-chains-trees:2} provides  information on the possible values of the prime powers $q_2, \dots, q_n$ allowed by the construction. We do not know if ascending chains of irreducible vertex-transitive lattices exist in the product $\Aut(T_{d_1}) \times \dots \times \Aut(T_{d_n})$ for \emph{all} $n$-tuple $(d_1, \dots, d_n)$ of integers $d_i \geq 3$. 
\end{remark}


\section{Bireversible automata}\label{sec:automata}

We shall now exploit the relation between lattices in products of two trees and bi-reversible automata that was introduced and described by Glasner--Mozes~\cite{GM05}. We review basic definitions and  refer  to the book \cite{Nekra} for detailed information on groups defined by automata. 

A \textbf{Mealy automaton} is a tuple $\mathcal M = (Q, L, \lambda, \rho)$ consisting of finite sets $Q$ and $L$ respectively called the \textbf{set of states} and the \textbf{alphabet}, and maps $\lambda  \colon Q \times L \to L$ and $\rho \colon Q \times L \to Q$ respectively called the \textbf{output map} and the \textbf{transition map}. The automaton $\mathcal M$ is called \textbf{invertible} if  for all $q \in Q$, the map  $ L \to L : \ell \mapsto \lambda(q, \ell)$ is a bijection. It is  called \textbf{reversible} if  for all $\ell \in L$ the map $Q \to Q : q \mapsto \rho(q, \ell)$ is a bijection. We say that $\mathcal M$ is   \textbf{bi-reversible} if it is both invertible and reversible, and if moreover the map $Q \times L \to  L\times Q : (q, \ell) \mapsto \big(\lambda(q, \ell), \rho(q, \ell) \big)$ is a bijection. 

If the Mealy automaton $\mathcal M$ is invertible, it defines a  group denoted by $G_{\mathcal M}$ that acts on the regular rooted tree, whose vertices are naturally in bijection with the free monoid generated by the alphabet $L$, see \cite[Def.~1.5.7]{Nekra} or \cite[Def.~1.1]{Fra23}. 

Glasner--Mozes have obtained the first examples of groups with Kazhdan's property (T) that are defined by a bi-reversible automaton, see \cite[Prop.~3.7]{GM05}. To achieve this, they first construct a lattice acting regularly on a product of two graphs (in their case, those are $1$-skeleta of Euclidean buildings of type~$\tilde A_2$), such that the stabilizer of a vertex in the second factor has property (T). Passing to the universal covering trees, one obtains an irreducible lattice in a product of trees, and the existence of the required bi-reversible automaton follows from general principles described in \cite{GM05}. The proof of Theorem~\ref{thm:bi-rev-autom} follows the same scheme; the only difference is that we use the vertex-regular lattices from  Proposition~\ref{prop:regular-action-prod} as a basic input, instead of a lattice in a product of two buildings of type~$\tilde A_2$. 
We shall actually provide an explicit description of the bi-reversible automata enjoying the properties asserted by Theorem~\ref{thm:bi-rev-autom}.  	

For the rest of this section, we fix integers $m_1, m_2 \geq 1$. 

Let  $q$ be a prime power with $q \geq m_1+m_2+2$, and $F$   be a finite field of order $q$. Let $F_1, F_2 \subset F$ be two disjoint subsets with $|F_i| = m_i+1$. For $i=1, 2$, we denote by $E_i$ the set of all  pairs of distinct elements of $F_i$. Hence $|E_i| = \frac{m_i(m_i+1)} 2$. We choose once and for all an ordering on each pair in $E_1$ and $E_2$. Thus we shall henceforth view the element of $E_i$ as ordered pairs of distinct elements of $F_i$.

For each $(a, b) \in E_i$, we choose a copy of $F$ that we denote by  $F_{(a, b)}$, and a bijection  $\sigma_{(a, b)} \colon F \to F_{(a, b)}$. We then set 
$$Q = \{ \sigma_{(a, b)}(x) \mid (a, b) \in E_1, \ x \in F\} = \bigsqcup_{(a, b) \in E_1} F_{(a, b)}$$
and
$$ L  =  \{ \sigma_{(c, d)}(y) \mid (c, d) \in E_2, \ y \in F\}= \bigsqcup_{(c, d) \in E_2} F_{(c, d)}.$$
We define a transition map $\rho  \colon Q \times L \to Q$ by setting 
$$
\rho\big( \sigma_{(a, b)}(x),   \sigma_{(c, d)}(y)\big) =  \sigma_{(a, b)}\left(\frac{a-d}{a-c} x + \frac{b-a}{a-c}y \right)
$$
and an output map $\lambda \colon Q \times L \to L$ by 
$$
\lambda\big( \sigma_{(a, b)}(x),   \sigma_{(c, d)}(y)\big) =  \sigma_{(c, d)}\left(\frac{c-d}{a-c} x + \frac{b-c}{a-c}y \right)
$$
for all $(a, b) \in E_1$, $(c, d) \in E_2$ and all $x, y \in F$. 
Finally, we let 
$$\mathcal B_{m_1, m_2; q} = (Q, L, \lambda, \rho)$$ 
be the  Mealy automaton defined by this $4$-tuple. 

We consider also the global field $K = F(t)$. For $i=1, 2$, we set 
$$S_i = \{v_{t+a} \mid a \in F_i\}$$ 
be the set of $f$-valuations of $K$ defined by the degree one polynomials $f=t+a$ with $a \in F_i$. We can now state the following result, that is a more precise version of Theorem~\ref{thm:bi-rev-autom} from the introduction. 

\begin{thm}\label{thm:bi-rev-autom:full}
Let $m_1, m_2 \geq 1$ be integers and $q$ be a prime power with $q \geq m_1+m_2+2$. Let $\mathcal B = \mathcal B_{m_1, m_2; q} $ be the automaton defined as above, and retain the associated notation. 
\begin{enumerate}[label=(\roman*)]
\item $\mathcal B$ is bi-reversible. 

\item The group $G_{\mathcal B}$ is isomorphic to the group $\Lambda_{S_1}$ from Proposition~\ref{prop:affine:Cayley}. In particular it is of type $\mathrm{F}_{m_1}$ but not of type $\mathrm{FP}_{m_1+1}$.

\item The group $G_{\mathcal B^*}$ is isomorphic to $\Lambda_{S_2}$.  In particular it is of type $\mathrm{F}_{m_2}$ but not of type $\mathrm{FP}_{m_2+1}$.
\end{enumerate}
	
\end{thm}

\begin{proof}
By definition, for all $(a, b) \in E_1$ and $(c, d) \in E_2$, the scalars $a, b, c, d$ are pairwise distinct. This implies that $\mathcal B$ is invertible and reversible. Moreover the determinant of the matrix 
$$\left(\begin{matrix}
a-d  &   b-a  \\
c-d  &  b-c
\end{matrix}
\right)$$
is non-zero. This implies that $\mathcal B$ is bi-reversible. Assertion~(i) follows. 

In order to prove the other assertions, we consider the fundamental group of the   square complex associated with $\mathcal B$ as in \cite{GM05}. It can be defined as the finitely presented group 
$$\Lambda_{\mathcal B} = \langle  Q \sqcup L \mid q\lambda(q, \ell) = \ell \rho(q, \ell) \ \forall (q, \ell ) \in Q \times L\rangle.$$  

Let $\Lambda_{S_1, S_2}$ be the affine arithmetic group afforded by Proposition~\ref{prop:regular-action-prod}, where we define the degree~one valuation $v_0 \not \in S_1 \cup S_2$ required by the definition to be $v_\infty$. The group $\Lambda_{S_1, S_2}$ acts regularly on the vertex set of $\mathsf{DL}_{m+1}(q) \times \mathsf{DL}_{n+1}(q)$. The stabilizer of a vertex in the second (resp. first) factor is isomorphic to $\Lambda_{S_1}$ (resp. $\Lambda_{S_2}$). In particular it has $\mathrm{F}_{m_1}$ but not $\mathrm{FP}_{m_1+1}$ (resp. type $\mathrm{F}_{m_2}$ but not $\mathrm{FP}_{m_2+1}$) by Theorem~\ref{thm:Bux}. 

We now construct a homomorphism $\varphi \colon\Lambda_{\mathcal B} \to \Lambda_{S_1, S_2} \leq K \rtimes K^*$ by setting 
$$\varphi \colon \sigma_{(a, b)}(x) \mapsto \big(\frac x {t+b}, \frac{t+a}{t+b}\big) \in \Lambda_{S_1}$$
for all $(a, b) \in E_1$ and all $x \in F$, and 
$$\varphi \colon \sigma_{(c, d)}(y) \mapsto \big(\frac y {t+d}, \frac{t+c}{t+d}\big) \in \Lambda_{S_2}$$
for all $(c, d) \in E_2$ and all $y \in F$. A direct computation in the affine group $K \rtimes K^*$ shows that those elements respect the defining relations of $\Lambda_{\mathcal B}$. Therefore the assignments above extend to a homomorphism $\varphi \colon\Lambda_{\mathcal B} \to \Lambda_{S_1, S_2}$ as claimed. 

The elements of $\varphi(Q \cup Q^{-1})$ (resp. $\varphi(L \cup L^{-1})$) map the base vertex of the Diestel--Leader graph $\mathsf{DL}_{m_1+1}(q)$ (resp. $\mathsf{DL}_{m_2+1}(q)$) associated with $\Lambda_{S_1}$ (resp. $\Lambda_{S_2}$) to each of its neighbors, see Remark~\ref{rem:explicit-gen-set}. Thus we have $\varphi(\langle Q\rangle) = \Lambda_{S_1}$ and $\varphi(\langle L\rangle) = \Lambda_{S_2}$. 

By Proposition~\ref{prop:regular-action-prod}, we know that the $\Lambda_{S_1, S_2}$-action on $\mathsf{DL}_{m_1+1}(q) \times \mathsf{DL}_{m_2+1}(q)$ is regular on the vertex set. 
Passing to the universal covering trees, we obtain a lattice acting regularly on the vertex set of a product $\mathcal T_1 \times \mathcal T_2$ of two trees, whose defining presentation is precisely that of $\Lambda_{\mathcal B}$, see  \cite[\S 2.6]{GM05}. The tree $\mathcal T_1$ is the Cayley tree of the free group with basis $Q$, and the tree $\mathcal T_2$ is the Cayley tree of the free group with basis $L$.  It follows that  the product $\Pi_1 \times \Pi_2$ of the respective fundamental groups of the graphs $\mathsf{DL}_{m_1+1}(q) $ and $\mathsf{DL}_{m_2+1}(q)$ coincides with the kernel of $\varphi$. 

In view of Theorem~\ref{thm:affine-arith:prod}(ii), the group $\Lambda_{S_1, S_2} $ acts faithfully on each factor of the product $\mathsf{DL}_{m_1+1}(q) \times \mathsf{DL}_{m_2+1}(q)$. This implies that $\Pi_1 \times \{e\}$ (resp. $\{e\} \times \Pi_2$) is the kernel of the $\Lambda_{\mathcal B}$-action on $\mathcal T_2$ (resp. $\mathcal T_1$).

We claim that the image of $\langle Q \rangle $ in $\Aut(\mathcal T_2)$ is isomorphic to $\Lambda_{S_1}$, and the image of $\langle L \rangle$ in $\Aut(\mathcal T_1)$ is isomorphic to $\Lambda_{S_2}$. Indeed, the group $\langle Q \rangle $  is the full stabilizer in $\Lambda_{\mathcal B}$ of a vertex of $\mathcal T_2$. Hence we have $\Pi_1 \leq \langle Q \rangle$. On the other hand $\Pi_2$ acts trivially on $\mathcal T_1$ and freely on $\mathcal T_2$, hence it intersects $\langle Q \rangle$ trivially. Therefore we have $\langle Q \rangle \cap (\Pi_1 \times \Pi_2) = \Pi_1 \times \{e\}$, and we obtain
\begin{align*}
\Lambda_{S_1} 
 & = \varphi(\langle Q \rangle) \\
 &\cong  \langle Q \rangle (\Pi_1 \times \Pi_2)/(\Pi_1 \times \Pi_2) \\
 & \cong  \langle Q \rangle  /\langle Q \rangle  \cap  (\Pi_1 \times \Pi_2)\\
 & = \langle Q \rangle  /\Pi_1\times \{e\}. 
\end{align*}

Since $\Pi_1 \times \{e\}$ is the kernel fo the $\Lambda_{\mathcal B}$-action on $\mathcal T_{2}$, this confirms the claim concerning $\langle Q \rangle $. The proof of the claim for $\langle L \rangle$ is similar. 

In view of the claim, it follows  from  \cite[Prop.~2.12]{GM05} that the automaton group $G_{\mathcal B}$ coincides with $\Lambda_{S_1}$ (resp. $G_{\mathcal B^*}$ coincides with $\Lambda_{S_2}$). This confirms the assertions (ii) and (iii). 
\end{proof}

\begin{remark}
For $m_1 = |S_1|=2$, the group $G_{\mathcal B} \cong \Lambda_{S_1}$ is a lamplighter group by Lemma~\ref{rem:lamplighter}. As mentioned in the introduction, the fact that lamplighter groups can be defined by bi-reversible automata is already known by the work of   \cite{BDR}, \cite{SkSt20} and \cite{Fra23}.
\end{remark}


\bibliographystyle{amsalpha} 
\bibliography{BMW}

\end{document}